\documentclass[10pt]{amsart}
\usepackage{etex}
\usepackage[french, english]{babel}
\usepackage{tikz}
\usetikzlibrary{matrix,arrows}
\usepackage[all,cmtip]{xy}
\usepackage{enumerate}
\usepackage{dsfont}
\usepackage{tikz-cd}
\usepackage{baskervald}
\usepackage{mathrsfs}
\usepackage{todonotes}
\usepackage{setspace}
\usepackage{longtable}
\usepackage{tabularx, tabu}
\usepackage{tikz}
\usepackage{amsfonts,amsthm,amsmath,amssymb,comment,stmaryrd,mathalfa,url,mathrsfs, yfonts}
\usepackage{graphicx}
\graphicspath{ {images/} }
\usepackage[section]{placeins}
\usepackage{microtype}
\usepackage{booktabs, calc, multirow}
\usepackage{array}
\usepackage[ocgcolorlinks,backref=page]{hyperref}
\hypersetup{citecolor=blue}
\usepackage{amscd}
\usepackage{makecell}
\usepackage{chngcntr}
\graphicspath{ {images/} }

\usepackage{adjustbox}
\usetikzlibrary{%
	matrix,%
	calc,%
	arrows%
}

\tikzset{
	labl/.style={anchor=north, rotate=90, inner sep=.5mm}
}

\def\AA{\mathbb{A}}
\def\CC{\mathbb{C}}

\def\QQ{\mathbb{Q}}
\def\RR{\mathbb{R}}
\def\ZZ{\mathbb{Z}}
\def \lra{\longrightarrow}

\def \a{\alpha}
\def \g {\gamma}

\def \l {\lambda}
\def \s {\sigma}
\def \u {u_{K/F}}

\def \Z {\wh{Z}}

\def \s {\sigma}
\def \W {W_{K/F}}
\def \HH {H_{1}(C_K,\wh{L})}
\def \ZH {Z_{1}(C_K,\wh{L})}
\def \HW {H_{1}(\W,\wh{L})}
\def \ZW {Z_{1}(\W,\wh{L})}

\def \wh{\widehat}
\def \H{\widehat{H}}
\def \L{\widehat{L}}

\def \G {\mathfrak{G}}

\def\F {\mathfrak{F}}

\DeclareMathOperator{\coinf}{Coinf}
\DeclareMathOperator{\tr}{Tr_{1}}
\DeclareMathOperator{\ttr}{Tr}
\DeclareMathOperator{\Tor}{Tor}
\DeclareMathOperator{\Cor}{Cor}

\DeclareMathOperator{\Gal}{Gal}
\DeclareMathOperator{\GL}{GL}
\DeclareMathOperator{\Hom}{Hom}

\DeclareMathOperator{\Res}{Res}

\DeclareMathOperator{\Inf}{Inf}

\theoremstyle{plain}
\newtheorem{thm}{Theorem}[subsection]
\newtheorem*{thm*}{Theorem}
\newtheorem{lem}[thm]{Lemma}

\newtheorem*{con*}{Conjecture}
\newtheorem{prop}[thm]{Proposition}
\newtheorem*{prop*}{Proposition}
\theoremstyle{definition}

\theoremstyle{definition}
\newtheorem{rmrk}[thm]{Remark}
\newtheorem*{rmrk*}{Remark}

\newtheorem*{exmp*}{Example}

\newtheorem*{obs*}{Observation}
\newtheorem{nota}[thm]{Notation}
\newtheorem*{nota*}{Notation}
\newtheorem{num}[thm]{\unskip}

\title{On the $p$-adic Langlands correspondence for algebraic tori}

	\author{Christopher Birkbeck}
\address{Department of Mathematics, University College London, Gower street, London, WC1E 6BT}
\email{c.birkbeck@ucl.ac.uk}

\begin{document}

\maketitle

\begin{abstract}We extend the results by R.P.  Langlands on representations of (connected) abelian algebraic groups. This is done by considering characters into any divisible abelian topological group. With this we can then prove what is known as the abelian case of the $p$-adic Langlands program.

\end{abstract}

\section*{\bf Introduction.}

In \cite{lang}, Langlands relates representations of the Weil group of a finite Galois extension $K$ of a number field $F$ into the $L$-group of an algebraic torus $T$, with  representations of $T(\AA_F) / T(F)$  into $\CC^\times$. The main goal of this paper is to extend these results by allowing representations of $T(\AA_F) / T(F)$ into a more general class of  groups. In particular, we want to look at representations  into $\CC_{p}^{\times}$ (the units of the completion of an algebraic closure of $\QQ_{p}$), which gives us what is called the abelian case of the $p$-adic Langlands program. These results are well-known to the experts but there appears to be no source in the literature for them. We have tried to stay faithful to the main ideas in Langlands' paper \cite{lang}, but aim to present the results in more detail and in more generality.

Before stating our main theorem let us setup some notation. Recall that there is one-to-one correspondence between algebraic tori defined over a field $F$, that split over a finite Galois extension $K$ of $F$, and equivalence classes of lattices on which $\Gal(K/F)$ acts. Here by lattice we mean a finitely generated free $\ZZ$-module (i.e. isomorphic to $\ZZ^{n}$ for some $n \in \ZZ_{ \geq 0}$). If $T$ is an algebraic torus and it corresponds to the lattice $L$, then the group $T(K)$ of $K$-rational points corresponds to the $\Gal(K/F)$-module $\Hom(L,K^{\times})$. Moreover, we have that $T(K)^{\Gal(K/F)} = T(F)$.

\begin{nota*}
	\begin{enumerate}
		\item  We let $F$ be any local or global field and $K$ a finite Galois extension of $F$   and let $W_{K/F}$ denote the relative Weil group as defined in \cite{tatewg}.
		\item Let $\wh{L}= \Hom(L,\ZZ)$ and  $\wh{T}_D=\Hom(\wh{L},D)$, where $D$ is any divisible abelian topological group with trivial $\Gal(K/F)$-action.
		\item Let $$ C_{E} = \begin{cases} \text{the idele class group of } E & \text{if }  E \text{ is a global field} \\
		E^{\times} & \text{if } E \text{ is a local field. }\end{cases}$$
		\item If $A,B$ are two topological groups, we let $\Hom_{cts}(A,B)$ represents the group of continuous group homomorphisms from $A$ to $B$, and similarly $Z_{cts}^{1}(A,B)$, $B_{cts}^{1}(A,B)$ represent the continuous 1-cocycles and 1-coboundaries, respectively, and $H_{cts}^{1}=Z_{cts}^{1} / B_{cts}^{1}$.	
	\end{enumerate}
\end{nota*}

Our main theorems are as follows:

\begin{thm*}There is a canonical isomorphism  \[H_{cts}^{1}(\W,\wh{T}_D) \overset{\sim}{\lra} \Hom_{cts}(\Hom_{\Gal(K/F)}(L,C_K),D).\]
	where $W_{K/F}$ has the usual topology as defined in \cite[Definition 1.1]{tatewg}, and $\wh{T}_D$ has the topology induced from $D$.	
\end{thm*}

If $F$ is a global field, we say an element in $H_{cts}^{1}(W_F,\wh{T}_{D})$ is locally trivial if it restricts to zero in $H_{cts}^{1}(W_{F_v},\wh{T}_{D})$ for all places $v$ of $F$.

\begin{thm*}
	\begin{enumerate}[(a)]
		\item If $K$ is a local field, then  $H_{cts}^{1}(\W,\wh{T}_D)$ is isomorphic to $\Hom_{cts}(T(F),D)$.
		\item Let $D'$ be a divisible abelian topological group such that for any finite group $G$, $\Hom(G,D')$ is finite, and let $$\wh{T}_{D'}=\Hom(\L,D').$$ 
		If $K$ is a global field, then there is a canonical  surjective homomorphism $$\H_{cts}^{1}(\W,\wh{T}_{D'}) \lra \Hom_{cts}(T(\AA_F) /T(F) , D'),$$and  the kernel of this homomorphism is finite. Furthermore, if $D'$ is also Hausdorff, then the kernel consists of the locally trivial classes.

	\end{enumerate}

\end{thm*}

A particular case of interest is when we set $D=\CC_p^{\times}$. It is well-known that, as fields, $\CC$ and $\CC_p$ are isomorphic, but not as topological fields. We will see that the topology on $D$ is irrelevant for the first theorem and is only needed for the last part of the second theorem, so since $\CC_p^{\times}$ is a divisible abelian topological group and  for any $n \in \ZZ_{>0}$ we have that the number of elements of order dividing $n$ is finite, which means $\Hom(G,\CC_p^{\times})$ will be a finite group for any finite group $G$, furthermore it is Hausdorff since it is a metric space. So we have that both theorems apply, from which we can then deduce the abelian case of the $p$-adic Langlands program.
In general, we can let $D=A^{\times}$, where $A$ is any Hausdorff topological field, like for example $\overline{\mathbb{F}}_{p},\overline{\QQ}_{p}$ or $\CC$. The case $D=\CC^{\times}$ was what was originally proved by Langlands in \cite{lang}.

The paper is split into three sections. In Section $1$ we setup some notation and recall some basic facts about Weil groups. Section $2$ contains the main technical results (which are largely, calculations in homological algebra), culminating in the proof of the first stated theorem in the introduction. Finally, in Section 3, we show how to use the results from Section $2$ prove second stated theorem.

\subsection*{Acknowledgements}I wish to thank Kevin Buzzard for suggesting this problem and encouraging me to write this paper. This work was part of the authors masters thesis under his supervision.

\section{\bf Setup and notation}

In what follows we will do many calculations involving cohomology classes in homology, cohomology and Tate cohomology groups. For this we well use the following notation:

\begin{nota}
	\begin{enumerate}
		\item Throughout, we will denote $\Gal(K/F)$ simply by $\G$.
		\item If $G$ is any group and $A$ is a $G$-module, then for $x \in Z_{n}(G,A)$ (a $n$-cocyle) we let $[x]$ represent its class in $H_{n}(G,A)$, and we do the same for cohomology and Tate cohomology groups. 
		\item Throughout, we will denote Tate cohomology group by $\wh{H}^i$. We recall that for $i\geq 1$, $\wh{H}^i=H^i$ and for $i < -1$, $\wh{H}^i=H_i$. So they only differ in degree $-1$ and $0$, where, for $G$ a finite group and $A$ a $G$-module, we have $\H^0(G,A)=A^G/N_G(A)$ and $\H^{-1}(G,A)=\ker(N_G)/I_{G}A$ where $N_G$ is the norm map\footnote{This sends $a$ to $\sum_{g\in G} ga.$} and $I_G$ is the augmentation ideal. Lastly, we let $A_{G}=A/I_{G}$ denote the coinvariants.
		
	\end{enumerate}\label{not1}
\end{nota}

\begin{rmrk}
	Note that from the action of $\G$  on $L$, we can define a $\G$-action on $\L$ by setting $(g\l)(x)=\l(g^{-1} \cdot x),$  for $\l \in \L$, $x \in L$ and $g \in \G$. Similarly, we define the action of $\G$ on $\wh{T}_D$  as $(g \a)(\l)=\a(g^{-1} \cdot \l)$ where $\a \in \wh{T}_D$, $\l \in \L$ and $g \in \G$.
\end{rmrk}

For our purposes we will use the following simple description of the Weil group $\W$.
\begin{prop}\label{wex}
	
	The Weil group $\W$ fits in an exact sequence $$0 \lra C_K \lra \W \overset{\s}\lra \G \lra 0,$$  corresponding to the fundamental class $[\u]$ in $H^{2}(G,C_K)$.
	
\end{prop}

\begin{proof} See \cite[(1.2)]{tatewg}. \end{proof}

\begin{nota}\label{ckreps}
	We fix once and for all $\{ w_g \mid g \in \G \}$ to be the set of left coset representatives of $C_K$ in $\W$.
	
\end{nota}
\begin{prop}\label{weill}
	Let $F$ be a global field and for each place $v$ of $F$ let $F_v$ denote the completion at $v$. Furthermore, let $W_F$ (resp. $W_{F_v}$) denote $W_{\overline{F}/F}$  (resp. $W_{\overline{F}_v/F_v}$ ) where $\overline{F}$ and $\overline{F}_v$ are fixed separable algebraic closures. Then we have a commutative diagram: 
	
	\[
	\xymatrix@C=4em@R=3em{
		W_{F_v}  \ar[r]^{\s_v} \ar[d] & \Gal(\overline{F}_{v}/F_v) \ar[d] \\
		W_F  \ar[r]^{\s} & \Gal(\overline{F}/F) \\
	}\]
	
\end{prop}

\begin{proof}
	See \cite[Proposition 1.6.1]{tatewg}.
\end{proof}

In what follows we will be concerned with representations of $\W$ into the group $${}^LT_{D}=\wh{T}_D \rtimes \G$$  (when $D=\CC^{\times}$, this group is known as the $L$-group of $T$). We want to study continuous homomorphisms $$\phi : \W \lra \wh{T}_D \rtimes \G$$ that make
$$
\xymatrix@C=2em@R=3em{
	\W \ar[r]^{\s} \ar[d]^{\phi} & \G \ar@{=}[d] \\
	\wh{T}_D \rtimes \G \ar[r] & \G \\
}$$
a commutative diagram; these are called admissible homomorphisms. Two admissible homomorphisms $\a,\beta$ from $\W$ to ${}^LT_{D}$ are called equivalent if there exists $t \in \wh{T}_D$ such that $\a = t \beta t^{-1}$. Now, note that we can write $\phi=f \times \s$, where $f \in Z_{cts}^{1}(\W,\wh{T}_D)$, from which it follows that two admissible homomorphisms $\a=f_{a} \times \s$ and $\beta=f_{\beta} \times \s$ from $\W$ to ${}^LT_{D}$ are equivalent if and only if $f_{\a}$ and $f_{\beta}$ represent the same cohomology class of $H_{cts}^{1}(\W,\wh{T}_D)$.

\section{\bf The duality theorem}
In this section we will prove the following:

\begin{thm}\label{t1} There is a canonical isomorphism  \[H_{cts}^{1}(\W,\wh{T}_D) \overset{\sim}{\lra}\Hom_{cts}(\Hom_{\Gal(K/F)}(L,C_K),D).\]
\end{thm}
We begin by proving that there exists an isomorphism $$\Psi: H^{1}(\W,\wh{T}_D) \lra \Hom(\Hom_{\Gal(K/F)}(L,C_K),D),$$ (this will follow from \ref{n1} and  Proposition \ref{t1p1}) and then we prove that $\Psi([f])$ is a continuous homomorphism if and only if $f \in Z_{cts}^{1}(\W,\wh{T}_D)$ (this is Proposition \ref{continu}).

In what follows we extend the natural action of $\G$ on $C_K$, to that of $\W$ on $C_K$, by letting $\W$ act by conjugation. Since $C_K$ is an abelian normal  subgroup of $\W$, we see that $C_K$ will act trivially on itself and hence we get an induced $\G$-action, which agrees with the standard Galois action of $\G$ on $C_K$. Also all $\G$-modules can be viewed as $\W$ modules, and therefore can also  be viewed  as $C_K$-modules, where $C_K$ will act trivially.

\begin{rmrk} Throughout we will be proving results about (co)homology groups and in the proofs we will always work with $n$-(co)cycles and usually ignore $n$-(co)boundaries, since in all of these cases the maps involved are maps between (co)homology groups which will automatically send (co)boundaries to (co)boundaries, so all that we need to check is how the maps in question act on the $n$-(co)cycles.
\end{rmrk}

\begin{prop}\label{p2} There is a natural $\G$-isomorphism of $H_{1}(C_K,\wh{L})$ with $\Hom(L,C_K).$
\end{prop}

\begin{proof}
	Since $C_K$ acts trivially on $\wh{L}$ then  $$H_1(C_K,\wh{L}) \cong C_K \otimes_{\ZZ} \wh{L}$$ since if a group $X$ acts trivially on a $X$-module $A$, then $H_{1}(X,A) \cong X/[X,X] \otimes_{\ZZ} A$ where $[X,X]$ denotes the commutator subgroup (see \cite[p.\ 164]{weib} ). So in this case, since $C_K$ is an abelian group we get the result above and note this will be a $\G$-isomorphism\footnote{Here $\G$ acts on 1-cycles  $x \in Z_1(C_K,\L)$ as $g \cdot x(a)= g x(g^{-1} \cdot a )$ for all $a \in C_K$, and $\G$ acts diagonally on $C_K \otimes \L$.}. Furthermore, we have a natural $\G$-isomorphism $C_K\otimes_{\ZZ} \L \to\Hom(L,C_K)$ where for $\wh{\l} \in \wh{L}$ and $a \in C_K$ we send  $a \otimes \wh{\l}$ to the homomorphism $\l \to a^{\langle \l, \wh{\l} \rangle}$ where $\l \in L$ and $\langle - , - \rangle$ is the natural bilinear paring $\langle - , - \rangle : L \times \wh{L} \to \ZZ$, and this is a $\G$-isomorphism.
	Combining these two isomorphisms, we get a $\G$-isomorphism 
	$$
	H_{1}(C_K,\wh{L})  \lra \Hom(L,C_K)$$
\end{proof}
Under this isomorphism we see that a 1-cycle  $x \in Z_{1}(C_K,\L)$, will map to the homomorphism 
$$ \lambda  \longmapsto \prod_{a \in C_K} a^{\langle \lambda, x(a) \rangle}, \qquad \text{for } \lambda \in L.
$$Note the homomorphism makes sense as the support of the 1-cycles and 1-boundaries is always finite.

Now, from the Universal Coefficients Theorem we have the following:

\begin{prop}\label{p3} Let $G$ be any group, and let $D$ be a divisible abelian group with trivial $G$-action. Then for all $n > 0$ we have an isomorphism $H^{n}(G,\Hom(B,D)) \to \Hom(H_{n}(G,B),D)$, for any left $G$-module $B$.
\end{prop}

\begin{proof}This follows from \cite[Corollary 7.61]{rotmanhom}.
\end{proof}

\begin{rmrk}\label{rmkiso} For $n=1$, the isomorphism from Proposition \ref{p3} can be seen to be induced by the  paring \[H^{1}(G,\Hom(B,D)) \times H_{1}(G,B)  \lra D ,\] which sends a $1$-cocycle $f$ and a $1$-cycle $x$ to
	$ \sum_{g \in G} \langle f(g) , x(g) \rangle.$\footnote{Here  by $\langle f(g) , x(g) \rangle$ we mean $f(g)(x(g))$.} 
	
\end{rmrk} 
\begin{num}\label{n1}
	We can now use this result to reduce the task of finding an isomorphism  \[\Psi: H^{1}(\W,\wh{T}_D) \lra \Hom(\Hom_{\Gal(K/F)}(L,C_K),D),\]   to finding an isomorphism  \[H_{1}(\W,\wh{L}) \lra H_{1}(C_K,\wh{L})^{\G},\] since once we have this, setting $n=1$, $B=\L$, and $G=\W$ in Proposition \ref{p3} and using Proposition \ref{p2} gives $\Psi$. To find this isomorphism we use the fact  $C_K$ is a normal subgroup of $\W$ and of finite index $|\G|$ together with the following:
\end{num}
\begin{prop}\label{Trans} Let $G$ be any group, and let $H$ be a subgroup of $G$ of finite index with $\{g_{i} \}$ denoting left coset representatives of $H$ in $G$. Then for $n \geq 0 $ and any $G$-module $A$, there exists unique homomorphisms $\ttr_n : H_{n}(G,A) \to H_{n}(H,A),$ such that:
	
	\begin{enumerate}
		\item For $n=0$, and all $a \in A$, we have $\ttr_{0}(\overline{a}) = \sum_{i} \overline{g_{i}a},$ where on the left $\overline{a}$ denotes the image of\/ $a$ in $A_G$ and on the right $\overline{g_{i}a}$ denotes the image of $g_{i}a$ in $A_H$. 
		
		\item If\/  $0 \overset{i}\lra A \lra B \overset{p}\lra C \lra 0$ is an exact sequence of $G$-modules, then there is a commutative diagram

		\adjustbox{scale=0.9, center} {	
			\xymatrix@C=4em@R=4em{
				H_{n}(G,B) \ar[r]^{p'}\ar[d]^{\ttr_{n}}  & H_{n}(G,C) \ar[r]^{\delta} \ar[d]^{\ttr_{n}} & H_{n-1}(G,A) \ar[r]^{i'} \ar[d]^{\ttr_{n-1}} & H_{n-1}(G,B) \ar[d]^{\ttr_{n-1}} \\
				H_{n}(H,B) \ar[r]^{p'}  & H_{n}(H,C) \ar[r]^{\delta} & H_{n-1}(H,A) \ar[r]^{i'} & H_{n-1}(H,B) 
			}
			
		}
		
	\end{enumerate}
\end{prop}
\begin{proof}
	See \cite[Proposition 9.93]{rotmanhom}. 
\end{proof}

Thus the transfer maps give us homomorphisms $\text{Tr}_{n}: H_{n}(\W,\wh{L}) \to H_{n}(C_K,\wh{L}).$ Our goal now is to prove:

\begin{prop}\label{t1p1}
	The map $\tr:\HW \to H_{1}(C_K,\L)^{\G}$ is an isomorphism.
\end{prop}

Since we are working with the idele class group $C_K$, the Tate--Nakayama Lemma (see \cite[Chapter IX, Section 8]{serrelocal} ) tells us that  we can use cup products to obtain isomorphisms between Tate cohomology groups\footnote{The conditions of the Tate--Nakayama Lemma hold by class field theory and the fact that $\Tor_1^{\ZZ}(\wh{L},C_K)=0$ since $\wh{L}$ is a free and hence flat $\ZZ$-module.}. So we can use Tate--Nakayama to form the following diagram (which will be referred to as ({\bf A}))

\adjustbox{max width=\linewidth,center}{%
	\begin{tikzcd}
	&&& \wh{H}^{-2}(\G,\wh{L}) \arrow[d,equal]\\ 
	& H_{1}(C_{K},\wh{L}) \arrow[d, twoheadrightarrow , "N_{\G}"] \arrow[r,"\Cor"] & H_{1}(W_{K/F},\wh{L}) \arrow[d, "\tr"] \arrow[r, "\coinf"] & H_{1}(\G,\wh{L}) \arrow[d, "\sim" labl] \arrow[r] & 0 & \\
	0 \ar[r]  & N_{\G} (H_{1}(C_{K},\wh{L})) \ar[r, hook] & H_{1}(C_{K},\wh{L})^{\G} \ar[r] & \wh{H}^{0}(\G,H_{1}(C_{K},\wh{L})) \arrow[r] & 0 & \\
	&&& \wh{H}^{0}(\G,\wh{L} \otimes C_K) \arrow[u, equal] \\ 
	\end{tikzcd}
}

Here the top sequence is derived from the standard Lyndon--Hochschild--Serre spectral sequence, the bottom sequence comes from the definition of the Tate cohomology groups and the fourth vertical arrow is given by taking cup products with the fundamental class $[\u]$. Since we will be trying to show this diagram commutes, it will be useful to recall how the maps involved are defined
\begin{itemize}
	\item If $x \in \ZH$, then $\Cor([x])$ is in the class containing the 1-cycle $y \in Z_{1}(\W,\wh{L})$ such that $y(w)=x(w)$ if $w \in C_K$ and $y(w)=0$ elsewhere. 
	\item If $x \in Z_{1}(\W,\wh{L})$, then $\coinf([x])=[y]$, where $y$ is the 1-cycle in $Z_{1}(G,\wh{L})$ such that $$ y(g) = \sum_{a \in C_K} x(aw_g),$$ where $w_g$ is as in \ref{ckreps}.  
\end{itemize}
Note that both of these maps will send cycles to cycles and boundaries to boundaries, so they are well-defined. Our goal is to first  show ({\bf A})  commutes. Once we have this, it follows at once (by a simple diagram chase)  that $\tr$ is surjective; we will then prove that $\tr$ is injective to finish the proof of Proposition \ref{t1p1}.

\subsection{$\tr$ is surjective}
Before proving surjectivity we first need to define $\tr$ and show that its image in in $\HH^{\G}$. To do this the strategy is to use dimension shifting and the definition of $\ttr_0$.

We begin by noting that for any $g \in \G$ and $w \in \W$, we have that  $w_g w \in \W$ belongs to a unique left coset of $C_K$ in $\W$. Therefore there is a unique element $u(w_g,w) \in C_K$ and a unique $j(g) \in \G$, such that $$w_g w = u(w_g,w)w_{j(g)},$$ where $j$ is just a permutation of the elements of $\G$. This can be related to the fundamental class $[\u] \in H^{2}(\G,C_K)$, by noting that the 2-cocycle $\u$ representing the fundamental class has the property that for each $g,g' \in \G$, $w_g w_{g'} = \u(g,g')w_{gg'}$. Therefore, if $w=aw_{g'} \in \W$, then $$u(w_g,w)=w_g a w_g^{-1} \u(g,g').$$

\begin{prop}\label{p5} If\/ $x \in Z_{1}(\W,\L)$, then, for all   $a \in C_K$, we have a well-defined map $$ \left ( \tr(x) \right ) (a) = \sum_{u(w_g,w)=a} w_g x(w), $$ Here the sum on the right is taken over all $g \in \G$ and $w \in \W$ such that $u(w_h,w)=a$. Furthermore, the  image of\/ $\tr$ is in $\HH^{\G}$.

\end{prop}

\begin{proof}
	We begin by considering the exact sequence $$0 \lra I_W \lra \ZZ[\W] \overset{\epsilon}\lra \ZZ \lra 0$$ where $\epsilon$ is the augmentation map $\sum_{i} n_i g_i \rightarrow \sum_{i} n_i$ and $I_W$ is the augmentation ideal from \ref{not1} (3). This is split over $\ZZ$, so it remains exact when tensored with $\L$, and thus we  get the exact sequence of $\W$ modules $$ 0 \lra I_W \otimes \L \lra \ZZ[\W] \otimes  \L \lra  \ZZ \otimes \L \lra 0.$$  Note that in this sequence we have $\W$ acting diagonally on each of the terms, but we can find a $\W$-module isomorphism that gives the middle term  an action only on the first term of the tensor product; this then makes $\ZZ[\W] \otimes \L$ an induced module, which means that, for any subgroup $S$ of $\W$, we have $$H_{n}(S,\ZZ[\W] \otimes \L )=0, \qquad \text{for } n>0.$$
	If we now identify $\ZZ \otimes \L$ with $\L$, then we get the exact sequence  $$ 0 \lra I_W \otimes \L \lra \ZZ[\W] \otimes  \L \lra \L  \lra 0$$ where the middle term, is an induced module, so we can  use dimension shifting to get a well-defined isomorphism $$\delta: \HW \overset{\sim}\lra H_{0}(\W, I_W \otimes \L),$$ that sends $[z] \in \HW$ to the class of $$ \sum_{w \in \W} (w^{-1}-1)(1 \otimes z(w)) \qquad  \text{(here the action is diagonal)}.$$ 
	Now by Proposition \ref{Trans} $(2)$, we get the following commutative  diagram
	
	$$
	\xymatrix@C=2em@R=2em{
		H_{1}(\W,\wh{L}) \ar[d]^{\tr} \ar[r]^-{\sim} & H_{0}(W_{K/F},I_W \otimes \wh{L}) \ar[d]^{\text{Tr}_0}\\
		H_{1}(C_{K},\wh{L})\ar[r]^-{\sim} & H_{0}(C_{K}, I_W \otimes \wh{L}).}
	$$
	with the horizontal isomorphisms given by $\delta$, defined above. Now take a 1-cycle $x \in \ZW$, under $\delta$, its image in $H_{0}(\W, I_W \otimes \L)$ is in the class of $\sum_{w \in \W} (w^{-1}-1)(1 \otimes x(w)).$ If we then apply $\ttr_{0}$, we get that it maps to the class of

	\begin{align*}
	\sum_{g,w} w_g \cdot (w^{-1}-1)(1 \otimes x(w)) & = \sum_{g,w} (w_g w^{-1}(1 \otimes x(w))-w_g(1 \otimes x(w)))\\
	& =  \sum_{g, w} w_{g} w^{-1}(1 \otimes x(w))-  \sum_{g,w} w_{g}(1 \otimes x(w)) \tag{$\blacklozenge$}
	\end{align*}
	
	in $H_{0}(C_K,I_W \otimes \L)$. Now, since we can write $w_g w= u(w_g,w)w_{j(g)}$, we can use this and the fact that $j$ is just a permutation of the elements of $\G$, to  write $(\blacklozenge)$ as \[\sum_{h,w} u(w_h,w)^{-1} w_h(1 \otimes x(w))- \sum_{g} \sum_{w} w_g(1 \otimes x(w))\]
	which after changing the summation index in the second term gives \[ \sum_{h,w} (u(w_h,w)^{-1}-1)w_h (1 \otimes x(w)).\]
	This can be rewritten as $$\sum_{a \in C_K} \left \{ (a^{-1}-1) \sum_{u(w_h,w)=a} w_h (1 \otimes x(w)) \right \}.$$
	Recall that we define the  action of $w_g \in \W$ on $a \otimes b$ as $w_g(a \otimes b)= w_g a \otimes w_g b$ and also note that \[(w_h-1)\otimes w_h x(w)=w_h(1 \otimes x(w))-(1 \otimes w_h x(w)),\] but the term on the left is clearly in $I_W \otimes \L$ so by definition of $H_{0}$ we have that the sum above is in  the same class  as  $$\sum_{a \in C_K} \left \{ (a^{-1}-1) \sum_{u(w_h,w)=a} 1 \otimes w_h  x(w) \right \},$$ in $H_{0}(C_K, I_W \otimes \L)$. But this is just the image under $\delta$ of the class of the 1-cycle $y \in \ZH$, where $y$ is defined as $$y: a \mapsto \sum_{u(w_h,w)=a} w_h  x(w).$$ Observe that this is indeed a 1-cycle, since $C_K$ acts trivially on $\L$, so $\sum_{a \in C_K} a^{-1}y(a)=\sum_{a \in C_{K}} y(a).$ Furthermore  it has finite support since $x$ has finite support.
	So by dimension shifting, it follows that $\tr([x])=[y]$.
	
	Lastly, we need to show that the image of $\tr$ is in $\HH^{\G}$, for which it suffices to show that for all $g \in \G$ and $x \in \ZW$, the class of $g \cdot \tr([x])$ is the same as the class of $\tr([x])$. From the definition of $\ttr_{0}$, it follows that the image of $\ttr_{0}$ is in $H_{0}(C_K,I_W \otimes \L)^{\G}$. Therefore we have the  following commutative diagram.

	$$
	\xymatrix@C=2em@R=2em{
		H_{1}(\W,\wh{L}) \ar[d]^{\tr} \ar[r]^-{\sim}  \ar `l[d] `[dd]_{g  \cdot \tr}  [dd] & H_{0}(W_{K/F},I_W \otimes \wh{L}) \ar[d]^{\text{Tr}_0} \ar `r[d] `[dd]^-{\ttr_0}  [dd]\\
		H_{1}(C_{K},\wh{L})\ar[r]^-{\sim} \ar[d]^{g \cdot } & H_{0}(C_{K}, I_W \otimes \wh{L}) \ar[d]^{g \cdot } \\
		H_{1}(C_{K},\wh{L})\ar[r]^-{\sim} & H_{0}(C_{K}, I_W \otimes \wh{L})
	}
	$$
	From this and \cite[Proposition 9.93]{rotmanhom}  we get that the class of $(g \cdot \tr)([x])$ is the same as the class of $\tr([x])$ for all $x \in \ZW$. So the image of $\tr$ is in $\HH^{\G}$.
	
\end{proof}

\begin{prop}\label{p51} The square 
	
	\begin{center}
		\begin{tikzcd}	
		H_{1}(C_{K},\wh{L}) \arrow[d]^{N_{\G}} \arrow[r, "\Cor"] & H_{1}(W_{K/F},\wh{L}) \arrow[d, "\tr"]\\
		N_{G}(H_{1}(C_{K},\wh{L})) \arrow[r, hook] & H_{1}(C_{K},\wh{L})^{\G} 
		\end{tikzcd}
	\end{center}
	
	is commutative.

\end{prop}

\begin{proof} We only need to show that $\tr \circ \Cor = N_{\G}$. This follows from  \cite[Chapter III, Proposition 9.5]{brown}), but the proof is simple so we include it for completeness. First note that, if $[x] \in  H_{1}(C_K,\L)$, then $\Cor([x])$ only has support in $C_K$, so $\tr(\Cor([x]))= [y]$, with $y \in Z_{1}(C_K,\L)$, such that $$y( a)= \sum_{u(w_h,w)=a} w_h  (\Cor(x))(w)= \sum_{w_h b w_h^{-1}=a} w_h x(b), $$ for $b \in C_K.$ Now we simply note that
	$ \sum_{h \in \G} h \cdot x(a)= (N_{\G}(x))(a).$

\end{proof}

Before continuing, we first need a way to express the action of taking cup products in terms of cycles and cocycles, for which we have the following three results.

\begin{lem}\label{l7} Let $G$ be a finite group and let $A,B$ be $G$-modules where we make the notational convention that for $a \in A$ with $N_G(a)=0$, we  let  $[a]^{0}, [a]_{-1}$ denote the canonical images of $a$ in $\H^{0}(G,A), \H^{-1}(G,A)$ respectively.

	Let $a \in A$ with $N_G(a)=0$ and $[r]$ is the class of a 1-cocycle $r \in Z^{1}(G,B)$. Then $$[a]_{-1} \cup [r] = [c]^{0},$$ where $$c =- \sum_{g \in G} ga \otimes r(g).$$

\end{lem}

\begin{proof}See \cite[ Lemma 2, p.\ 176--177.]{serrelocal}. 
\end{proof}

This lemma is just what we need to be able to express the action of taking cup products in terms in cycles and cocycles. Recall that for any finite group $G$ and  $G$-module $B$  we can form the exact sequence $$0 \lra I_G \otimes B \lra \ZZ[G] \otimes B \lra B \lra 0 $$ as we did in Proposition \ref{p6}. Similarly, since the category of $G$-modules has enough injectives, we can find $G$-modules $B',B''$ such that $$ 0 \lra B \lra B' \lra B'' \lra 0$$ is an exact sequence and $B'$ is an induced module. From which we get isomorphisms 
\begin{align*}
& \delta : \H^{n}(G,B) \overset{\sim}{\lra} \H^{n+1}(G, I_G \otimes B);\\
&\partial : \H^{n}(G, B'')   \overset{\sim}{\lra}\H^{n+1} (G,B).
\end{align*}
In particular,  when $n=-1$ the isomorphisms are induced by $N_G$.  So we get
\begin{align*}
&\delta: \H^{-1}(G,B) \lra \H^{0}(G, I_G \otimes B)\\
&[y]   \longmapsto \left [ \sum_{g \in G} g \otimes g.y  \right ] \qquad 
\end{align*}
and
\begin{align*}
\partial: \H^{-1}(G,B'') &\lra \H^{0}(G,  B)\\
[x] &  \longmapsto \left [ \sum_{g \in G}  g. x \right ].
\end{align*}

\begin{prop}\label{p8} Let $G$ be a finite group, and let $A,B$ be $G$-modules. If $f \in \Z^{1}(G,A)$ is a 1-cocycle and $x \in \Z^{-2}(G,B)$ is a 1-cycle, then $[x] \cup [f]$ is in the class of  $$F:=- \sum_{g \in G} x(g) \otimes f(g) $$ in $\H^{-1}(G,B \otimes A).$

\end{prop}

\begin{proof} First note that since we have an isomorphism $\H^{-1}(G,B \otimes A) \cong \H^{0}(G,I_G \otimes B \otimes A)$ induced by $N_G$, which we denote by $\delta$. In order to prove the result, it is enough to check that  the class of  $\delta([x] \cup [f])=\delta([x]) \cup[f]$ in $\H^{0}(G,I_G \otimes B \otimes A)$ is the class containing $$\delta(F)= - \sum_{g,h \in G} h \otimes h x(g) \otimes h f(g).$$ Now, as before, we have that under $\delta$ the image of $[x]$ in $\H^{-1}(G,I_G \otimes B)$ is in the class of $$b = \sum_{g \in G} (g^{-1}-1)(1 \otimes x(g)).$$ Note that $$ N_G(b)=\sum_{g,h \in G} (hg^{-1}-h)(1 \otimes x(g))=\sum_{g,h} hg^{-1}(1 \otimes x(g))- \sum_{g,h}h(1 \otimes x(g)).$$ Since we are working in $I_G \otimes B$ we have that this can be rewritten as $\sum_{g,h} 1 \otimes hg^{-1} x(g) - \sum_{g,h}1 \otimes hx(g).$ But now recall that since $x$ is a 1-cycle, we have $\sum_{g \in G} g^{-1}x(g) = \sum_{g} x(g),$ which combined with the above, tells us that $N_G(b)=0$. Therefore we can apply Lemma \ref{l7} with $A$ and $B$ replaced by $I_G \otimes B$ and $A$ respectively,
	to get that the class of $[b] \cup [f]$ in $\H^{0}(G,I_G \otimes B \otimes A)$ is the class containing
	
	\begin{align*}
	-\sum_{h} h.b \otimes f(h)&=- \sum_{h,g} hg^{-1} \otimes hg^{-1}x(g) \otimes f(h) + \sum_{h,g} h \otimes hx(g) \otimes f(h) \\
	&=- \sum_{h,g } h \otimes hx(g) \otimes f(hg) + \sum_{h,g } h \otimes hx(g) \otimes f(h). 
	\end{align*}
	However, since $f$ is a 1-cocycle we have that $f(hg)=f(h)+hf(g)$, which after substituting gives $$- \sum_{g,h \in G} h \otimes h x(g) \otimes h f(g).$$

\end{proof}

Now we can use this to get  a result for 2-cocycles.

\begin{prop} \label{p9}Let $G$ be a finite group and let $A,B$ be $G$-modules. If $f \in \Z^{2}(G,B)$ is a 2-cocycle and $x \in \Z^{-2}(G,A)$ is a 1-cycle, then the class of $[f] \cup [x]$ in $\H^{0}(G,B \otimes A)$ is the class containing $$\sum_{g,h \in G} f(g,h) \otimes g x(h)$$

\end{prop}
Note that, in this case, $[f] \cup[x] =[x] \cup [f]$.

\begin{proof}(Based on J.P Serre \cite[Lemma 4, p.\ 178.]{serrelocal}) We begin by noting that since we have an exact sequence $0 \to B \to B' \to B'' \to 0,$ with $B'$ an induced module, then  $H^{2}(G,B')=0$. This means we can find a 1-cochain $f':G \to B'$, such that $$f(g,h)=gf'(h)-f'(gh)+f'(g).$$
	If we compose $f'$ with the map $B' \to B''$, we get  a 1-cocycle $f'' :G \to B''$, such that $\partial([f''])=[f]$.
	We can then use this and the previous proposition to see that 
	\begin{align*} 
	[\overline{f} \cup \overline{x}] =  [\partial([f'']) \cup [x] ]=\delta ( [[f''] \cup [x]) \overset{(\ref{p8})}{=\joinrel=}&\delta \left ( \left [ \sum_{h \in G} f''(h) \otimes x(h) \right ] \right )\\ =& \left [ \sum_{g,h \in G} g \cdot f'(h) \otimes g x(h) \right ].\tag{$\dagger$} 
	\end{align*}
	(In the last equality we change from $f''$ to $f'$ since by definition of $\delta$ we must first lift to $B'$.)
	
	Now, we know that $g \cdot f'(h) =f(g,h) + f'(gh) - f'(g),$ so $(\dagger)$ becomes the class containing \[\sum_{g,h \in G} \left \{ f(g,h) + f'(gh) - f'(g) \right \} \otimes g x(h)\] which we expand as \[\sum_{g,h} f(g,h) \otimes g x(h) + \sum_{g,h} (f'(gh)-f'(g)) \otimes gx(h).\] Therefore, in order to finish the proof we have to show that the second term is actually zero. By changing summation indexes we have \begin{align*}\sum_{g,h} (f'(gh)-f'(g)) \otimes gx(h)&=\sum_{g,h} f'(gh) \otimes gx(h) - \sum_{g,h} f'(g) \otimes gx(h)\\ &=\sum_{g,h} f'(g) \otimes gh^{-1} x(h) - \sum_{g,h} f'(g) \otimes gx(h)\\
	&=\sum_{g} f'(g) \otimes g \left ( \sum_{h} (h^{-1}-1)x(h) \right )  =0. \end{align*} with the last equality due to $x$ being a $1$-cocycle.
\end{proof}
Now with this result we can prove:

\begin{prop}\label{p6} The square 
	
	\centerline{
		\xymatrix@C=2em@R=2em{
			H_{1}(\W,\wh{L}) \ar[d]^{\tr} \ar[r]^{\coinf} & H_{1}(\G,\wh{L}) \ar[d]^{\cup [\u]}\\
			H_{1}(C_{K},\wh{L})^{\G} \ar[r]& \H^{0}(\G,\Hom(L,C_K))
	}}
	is commutative.

\end{prop}
Note that once we have the commutativity of this square we will at once have that diagram ({\bf{A}} is commutative (after using the isomorphism $\HH \cong \Hom(L,C_K)$).

\begin{proof}
	
	We begin by taking a 1-cycle $x \in \ZW$. From Proposition \ref{p5}, its image in $\HH^{\G}$ under $\tr$ is in the class of the 1-cycle $$y: a \longmapsto \sum_{u(w_h,w)=a}  w_h  x(w).$$ As before the sum is taken over all $h \in \G$ and $w \in \W$ such that $u(w_h,w)=a$. 
	Under the isomorphism of Proposition \ref{p2} and the natural map  $\HH^{\G} \to \H^{0}(\G,\HH)$, we get that the image $[y]$  in $\H^{0}(\G,\Hom(L,C_K))$ is in the class containing the homomorphism $$\l \longmapsto \prod_{a \in C_K} a^{\langle \l ,y(a) \rangle}= \prod_{a} \prod_{u(w_h,w)=a} a^ { \langle \l, w_h x(w) \rangle}=\prod_{h,w} u(w_h,w)^{ \langle \l, w_h x(w) \rangle}.$$
	Since each $w \in \W$ can be written as $aw_g$ for some $g \in \G$ and $a \in C_K$,  we have $$u(w_h,w)=w_h a{ w_h}^{-1} \u(h,g).$$ So we can rewrite the homomorphism above as $$ \l \longmapsto \left \{ \prod_{g,h,a} ( w_h a {w_h}^{-1} )^{\langle \l , w_h x(aw_g) \rangle} \right \} \left \{ \prod_{g,h,a} \u(h,g)^{\langle \l , w_h x(aw_g) \rangle} \right \}.$$
	Now note that the first term is the image of a norm, so since we are working in the zeroth Tate cohomology group, we get that this homomorphism is in the same class as $$ \beta : \l \longmapsto  \prod_{g,h,a} \u(h,g)^{\langle \l , w_h x(aw_g) \rangle} =\prod_{g,h} \u(h,g)^{ \langle \l , h z(g) \rangle},$$ where $$z(g)=\sum_{a \in C_K} x(aw_g), \qquad \text{for all } g \in \G.$$
	Alternatively, if we take $x \in \ZW$ and go along the top of the square we have that $\coinf([x])=[z]$. So in order to show the square commutes we must show that $[[z] \cup [\u]] = [\beta]$.
	
	Now let $B=C_K$ and $A=\L$ in Proposition \ref{p9}, then we can take $f(h,g)=\u(h,g)$ so that $[\u]$ is the fundamental class, and we take $x$ to be $z$. Then Proposition \ref{p9} tell us that the class of $[z]\cup [\u]$ in $\H^{0}(\G, \L \otimes C_K)$ is $$\sum_{g,h \in \G} hz(g) \otimes \u(h,g),$$ which  maps to the homomorphism $$\beta: \l \longmapsto \prod_{g,h \in \G} \u(h,g)^{\langle \l , h z(g) \rangle}$$ in $\H^{0}(\G,\Hom(L,C_K))$  as required.
	
\end{proof}

Thus we have that the second square in diagram ({\bf{A}}) commutes. As we mentioned before, this now tells us that $\tr$ is surjective (this is just a simple diagram chase). We are left proving that $\tr$ is injective. 

\subsection{$\tr$ is injective}
Note that, from the commutativity of ({\bf A}), its enough to show that the kernel of the map $$\Cor: H_{1}(C_K,\L) \lra H_{1}(\W,\L)$$ is equal to the kernel of the map $$N_{\G}: H_{1}(C_K,L) \lra H_{1}(C_K,\L).$$ In other words, we want to show that the kernel of the corestriction consists precisely of the elements of norm zero. This is equivalent to showing that the the image of $$\Cor':\Hom(\HW,\QQ/\ZZ) \lra \Hom(\HH,\QQ/\ZZ),$$ consists of homomorphisms that vanish on elements of norm zero\footnote{Note that $\Hom(-,\QQ/ \ZZ)$ is an exact contravariant functor since $\QQ/ \ZZ$ is divisible.}. Here $\Cor'$ denotes the map induced by $\Cor$. Now using Proposition \ref{p3}, we have isomorphisms
\begin{align*}
\F': \Hom(\HW,\QQ/\ZZ) &  \overset{\sim}{\lra} H^{1}(\W, \Hom(\L, \QQ/\ZZ))\\
\F: \Hom(\HH,\QQ/\ZZ) &  \overset{\sim}{\lra} H^{1}(C_K,\Hom(\L,\QQ /\ZZ)).
\end{align*}
So, its enough to show that the image of $$ H^{1}(\W, \Hom(\L, \QQ/\ZZ)) \lra H^{1}(C_K,\Hom(\L,\QQ /\ZZ)),$$ consists of elements corresponding (under $\F$) to homomorphisms that vanish on elements of norm zero. In other words, we want to show that given $[\psi] \in H^{1}(C_K, \Hom(\L, \QQ/\ZZ))$, we can extend this to a $[\Psi] \in  H^{1}(\W, \Hom(\L, \QQ/\ZZ))$ if and only if $[\psi] = \F(\varphi)$, where $\varphi$ is a homomorphism vanishing on elements of norm zero. Following Langlands \cite[ p.13]{lang}, we reformulate this problem as follows:
recall from Proposition \ref{wex} we have the following exacts sequence $$0 \lra C_K \lra \W \overset{\s}\lra \Gal(K/F) \lra 0,$$ whose class in $H^{2}(\G,C_K)$ is $[\u]$ (the fundamental class). Also, we can use the action of  $\G$  on $\L$, to give $\Hom(\L,\QQ / \ZZ)$ a $\G$-action, by letting $\G$ act trivially on $\QQ / \ZZ$, and  we can form the semi-direct product $\Hom(\L, \QQ /\ZZ) \rtimes \G.$ Now suppose we have the following commutative diagram 
$$
\xymatrix@C=2em@R=2em{
	0 \ar[r] & C_K \ar[r] \ar[d]^{\psi} & \W  \ar[r]^{\s} \ar[d]^{\Psi} & \G \ar[r]\ar[d]^{id} & 0\\
	0 \ar[r]  & \Hom(\L, \QQ /\ZZ) \ar[r] &\Hom(\L, \QQ /\ZZ) \rtimes \G  \ar[r] & \G \ar[r] & 0. 
}
$$
We can define a 1-cochain $f$ by $\Psi(w) = f(w) \times \s(w)$ and, in fact,  $f \in Z^{1}(\W, \Hom(\L, \QQ/\ZZ))$. Conversely, given $f \in Z^{1}(\W, \Hom(\L, \QQ/\ZZ))$ we can define $\Psi = f \times \s$, such that $\Psi$ together with its restriction $\psi$  to $C_K$ will make the above diagram commute. Now what we want to prove is that given a homomorphism $$\psi: C_K \lra  \Hom(\L, \QQ /\ZZ),$$ we can extend this to a homomorphism $$\Psi: \W \lra \Hom(\L, \QQ /\ZZ) \rtimes \G$$ making the diagram commute if and only if $\psi$ corresponds (under $\F$) to a homomorphism $\varphi: H_{1}(C_K,\L) \to \QQ / \ZZ$ that vanishes on elements of norm zero. By \cite[Theorem 2]{artintate2}  $\psi$ will extend to $\Psi$ if and only is $\psi$ is $G$-invariant and\footnote{This is because the class representing a semi-direct product is zero.} $\psi_{*}([\u]) =0 $, where $\psi_{*}$ is the map induced by $\psi$.

Now, if $\psi: C_K \to  \Hom(\L, \QQ /\ZZ),$  corresponds (under $\F$) to  \[\varphi: H_{1}(C_K,\L) \to \QQ / \ZZ,\] then it is easy to see that for all $a \in C_K$ and $\wh{\l} \in \L$, we have $\left (\psi(a) \right) (\wh{\l}) = \varphi( \wh{\l} \otimes a)$, where we are using the fact that $H_{1}(C_K,\L) \cong \L \otimes C_K$. So we want to show that $\psi$ is $\G$-invariant if $\varphi$ vanishes on elements of norm zero. But, by the above, we see that $\psi$ is $\G$-invariant, if and only if for all $\wh{\l} \in \L$, $ a \in C_K$ and all $g \in \G$, we have $\psi(g \cdot a)(\wh{\l})=\psi(a)(g^{-1} \wh{\l}),$ which is true if and only if  $\varphi( \wh{\l}  \otimes ga) = \varphi( g^{-1} \wh{\l}  \otimes a).$ Now the latter will hold if for all $a \otimes \wh{\l} \in C_K \otimes \L$, we have $$g\cdot ( \wh{\l}  \otimes a)- (\wh{\l}  \otimes a) \in \ker(\varphi)$$ which is true if $\varphi$ vanishes on elements of norm zero. 

So it remains to prove that $\psi_{*}( [\u]) =0 $. For each pair of homomorphisms $\psi : C_K \to \Hom(\L,\QQ/\ZZ)$ and $\varphi: H_{1}(C_K,\L) \to \QQ / \ZZ$ with $\psi = \F(\varphi)$, we have the following commutative diagram$$
\xymatrix@C=4em@R=2em{
	\L \otimes C_K \ar[r]^-{id \otimes \psi} \ar[d] & \L \otimes \Hom(\L,\QQ / \ZZ) \ar[d] \\
	H_{1}(C_K,\L) \ar[r]^{\varphi} &  \QQ / \ZZ
}$$
where the first vertical arrow is the isomorphism (see \ref{p2}) that sends $\wh{\l} \otimes a$ to the 1-cycle that is zero except at $a$ where it is $\wh{\l}$, for $\wh{\l} \in \L$ and $a \in C_K$. Now, if $\psi$ (and hence $\varphi$) is $\G$-invariant we get the following commutative diagram:$$
\xymatrix@C=4em@R=2em{
	\H^{-3}(\G,\L) \otimes \H^{2}(\G,C_K) \ar[r]^-{id \otimes \psi_{*}} \ar[d]^{\mu} & \H^{-3}(\G,\L) \otimes \H^{2}(\G,\Hom(\L,\QQ / \ZZ)) \ar[d]^{\nu} \\
	\H^{-1}(\G,H_{1}(C_K,\L)) \ar[r]^{\varphi'} &  \H^{-1}(\G,\QQ / \ZZ),
}$$
where the vertical arrows are given by taking cup products, the map $\varphi'$ is induced by $\varphi$. Note that for $\g \in \H^{-3}(\G,\L)$, the map $\g \to \mu(\g \otimes [\u]),$ gives an isomorphism $$\H^{-3}(\G,\L) \lra \H^{-1}(\G,H_{1}(C_K,\L)).$$ This is the isomorphism given by taking cup-product with the fundamental class $[\u]$. Similarly, for $\beta \in \H^{-3}(\G,\L)$ and $\g \in  \H^{2}(\G,\Hom(\L,\QQ / \ZZ))$  the map $\beta \to \nu(\beta \otimes \g)$ is an isomorphism for all $\g \in  \H^{2}(\G,\Hom(\L,\QQ / \ZZ))$. For a proof of this see \cite[Corollary 7.3]{brown} and note that in this case, both  $\H^{-3}(\G,\L)$ and  $\H^{2}(\G,\Hom(\L,\QQ / \ZZ))$ are finite groups  (see \ref{p11}). So, from the diagram, we see that $\psi_{*}([\u]) = 0 $ if and only if for all $\beta \in \H^{-3}(\G,\L)$ $$\nu(\beta \otimes \psi_{*}([\u]))=0.$$  Now going around the diagram in the other direction we see that, since the map $\g \to \mu(\g \otimes [\u])$ is an isomorphism, we have that $\psi_{*}([\u])=0$ if and only if $\varphi'$ is the zero map. This, by definition of $\varphi$ and $\H^{-1}$, is true if and only if $\varphi$ vanishes on elements of norm zero. This now completes the proof that $\tr$ is injective and thus we have proven Proposition \ref{t1p1}.

\subsection{Continuity} 

Note here that by using Proposition \ref{p2}, we can give $H_{1}(C_K,\L)$ a topology by using the natural topology on $\Hom(L,C_K)$, and consequently we get a topology on $\HW$ by using the fact that $\tr$ is an isomorphism. Our goal now is to prove the following:

\begin{prop}\label{continu}
	If\/  $x \in Z^{1}(\W,\wh{T}_D)$, then \[\Psi([x])  \in \Hom_{cts}(\Hom_{\G}(L,C_K), D)\] if and only if\/ $x$ is a continuous 1-cocycle.
\end{prop}
The proof will require several results. We begin with some basic lemmas.

\begin{lem}{\label{fg}} If $M$ is a finitely generated $G$-module and $G$ is a finite group, then $\wh{H}^{i}(G,M)$ is a finite group.

\end{lem}
\begin{proof}
	See \cite[ Proposition (3-1-9)]{weiss}.
\end{proof}

\begin{lem}\label{buzz}
	Let $A$ be a topological group and let $H$ and $S$ be subgroups of $A$ with $H$ open in $A$. If\/ $H \cap S$ is closed in $H$ then $S$ is closed in $A$.

\end{lem}

\begin{proof} This is elementary.
	
	
	
\end{proof}

With this we now have the following proposition:

\begin{prop}\label{p11}
	For all $i \in \ZZ$, $\wh{H}^{i}(\G,\Hom(L,C_K))$ is a finite group.

\end{prop}
\begin{proof}
	
	Since we have a $\G$-module isomorphism between $\Hom(L,C_K)$ and $\wh{L} \otimes C_K$, it is enough to prove that $\wh{H}^{i}(\G,\wh{L} \otimes C_K)$ is a finite group. Now from the Tate--Nakayama Lemma we have that for all $i \in \ZZ$ $$\H^{i}(\G,\L) \cong \H^{i+2}(\G,\L \otimes C_K).$$  So we can reduce the problem to showing that $\H^{i}(\G,\L)$ is finite. But since $L$ is a finitely generated $\ZZ$-module, $\L$ will also be a finitely generated $\ZZ$-module and consequently $\L$ will also be a finitely generated $\ZZ[\G]$-module, so  Lemma \ref{fg} applies, giving the result.

\end{proof}

Next we have the following key result.
\begin{prop}\label{p10}
	A homomorphism $\a: \Hom_{\G}(L,C_K) \to D$ is continuous if and only if $\a \circ N_{\G}$ is continuous.
\end{prop}
In order to prove this, it suffices to prove that $N_{\G}(\Hom(L,C_K))$ (the image under $N_{\G}$) is an open subgroup of $\Hom_{\G}(L,C_K)$, since a homomorphism of topological groups is continuous if and only if it is continuous at the identity. It follows that a homomorphism will be continuous on $\Hom_{\G}(L,C_K)$ if and only if it continuous on an open subgroup.  In particular, using this result and Proposition \ref{p2}, we may replace $\Hom_{\G}(L,C_K)$ with $\Hom(L,C_K) \cong H_1(C_K,\L)$ in Proposition \ref{continu}.  

Applying Proposition \ref{p11} with $i=0$, gives that $N_{\G}(\Hom(L,C_K))$ has finite index in $\Hom_{\G}(L,C_K)$. Therefore in order to prove Proposition \ref{p10}  it suffices to prove that  $N_{\G}(\Hom(L,C_K))$ is closed in $\Hom_{\G}(L,C_K)$, since any closed subgroup of finite index is automatically open. Now observe that if $K$ is a local field or a global function field, then we have a natural homomorphism from $C_K$ into $\ZZ$, whose kernel $U_K$ is known to be compact. Similarly, if $K$ is a  number field, then there is a natural map from $C_K$ to $\RR^{>0}$, whose kernel we once again denote by $U_K$ and is also compact (see \cite[Theorem 1.6]{neukalg}). With this we can form the exact sequence of abelian groups \begin{equation} 1 \lra U_{K} \lra C_{K} \lra M_K \lra 1, \tag{$*$} \end{equation} where we set $M_K=\ZZ$ or $M_K=\RR^{>0} \cong \RR$ accordingly, and we call the two cases the ``$\ZZ$-Case'' and  ``$\RR$-Case'' respectively.

\subsubsection{\bf $\ZZ$-Case} Since $L$ is a free $\ZZ$-module (hence projective) we can use $(*)$ to form the exact sequence 
$$
0 \lra \Hom(L,U_K) \overset{\lambda}\lra \Hom(L,C_K) \overset{\mu}\lra \Hom(L,M_K) \lra 0,
$$
where we think of this as a sequence of $\G$-modules by giving  $M_K$ the trivial action. Note that in the $\ZZ$-case we have  $$\H^{i}(\G,\Hom(L,M_K))=\H^{i}(\G,\L), \qquad \text{for all }i \in \ZZ$$ and Lemma \ref{fg} tells us that all of these groups are finite.

\begin{prop}\label{p13} There is an injective map $\psi$, from  \[\left ( N_{\G}(\Hom(L,C_K)) \cap \Hom(L,U_K) \right ) /N_{\G}(\Hom(L,U_K))\] to \[\wh{H}^{-1}(\G,\Hom(L,M_K))/ \mu \wh{H}^{-1}(\G,\Hom(L,C_K)).\]
\end{prop}
In order to ease notation in the proof, we set  \[B= N_{\G}(\Hom(L,C_K)) \cap \Hom(L,U_K)\] and \[V=\wh{H}^{-1}(\G,\Hom(L,M_K))/ \mu \wh{H}^{-1}(\G,\Hom(L,C_K)).\] Note that once we have this result, it will follow that$N_{\G}(\Hom(L,U_K))$ has finite index in $B$, since $V$ is finite by the comment above.

\begin{proof}
	We begin by taking  $x \in \Hom(L,C_K)$  such that $z=N_{\G}(x) \in \Hom(L,U_K)$.  If we let $y= \mu(x)$ with $\mu$ as above, then by exactness and the fact that $\mu$ is a homomorphism we have $$N_{\G}(y)=N_{\G}(\mu(x))=\mu(N_{\G}(x))=0.$$
	
	We claim there is a well-defined map $\psi$, such that $\psi(z)= \overline{y}$ where $\overline{y}$ is the image of $y$ in $V$. Note that if $x \in \Hom(L,U_K)$, then $\overline{y}$ will be zero. To prove this claim, observe that the image of $y$ will be independent of $x$ since on the right we quotient out by $\mu \wh{H}^{-1}(\G,\Hom(L,C_K))$. Therefore, if we had $z=N_{\G}(x)=N_\G(x'),$ then letting $x-x'=r$, we would have $N_\G(r)=0$ and $$y=\mu(x)=\mu(x')+\mu(r)=y'+\mu(r).$$ So when we look at $\overline{y}$ and $\overline{y}'$  we can clearly see they will represent the same element in $V$, hence $\psi$ is well-defined. 
	
	In order to show the $\psi$ is injective, it suffices to show that if $\psi(z)=0$ for $z=N_\G(x)$, and $x \in \Hom(L,C_K)$, then we can chose an $x' \in \Hom(L,U_K)$ such that $N_\G(x)=N_\G(x')$. So suppose that $\psi(N_{\G}(x))=\psi(z)=0,$ then we have $\mu(x)=y \in I_\G(\Hom(L,M_K))$ (by definition of $\H^{-1}$), hence $$y = \sum_{g} (g^{-1}-1)v_{g}, \qquad \text{ for some } v_{g} \in \Hom(L,M_K).$$Now since $\mu$ is surjective we can pick $u_g \in \Hom(L,C_K)$, such that $\mu(u_g)=v_g$ and we can also pick $x \in \Hom(L,C_k)$ such that $\mu(x)=y$. Now consider $$x'=x-\sum_{g} (g^{-1}-1)u_g,$$ it must lie in $\Hom(L,U_K)$ since $\mu(x')=\mu(x)-\sum_{g} (g^{-1}-1)v_g=0$, but we also have $N_\G(x)=N_\G(x')$, so we are done.
\end{proof}
With this result we can now show that in the $\ZZ$-case, $ N_\G(\Hom(L,C_K))$ is closed in $\Hom_{\G}(L,C_K)$. First note that $L \cong \ZZ^n$ as abelian groups (for some $n \in \ZZ_{ \geq 0}$), and  since $U_K$ is compact and Hausdorff, then $\Hom(L,U_K) \cong (U_{K})^n$ is also compact and Hausdorff  (being the direct sum of compact and  Hausdorff groups). Also since $N_\G$ is a continuous map, we have that $N_\G(\Hom(L,U_K))$ must be closed in $\Hom(L,U_K)$  (being a compact subgroup of a Hausdorff group). Therefore, since  $B$ is a subgroup of $\Hom(L,U_K)$ and $N_\G(\Hom(L,U_K))$ is closed in $\Hom(L,U_K)$, we must have that $N_\G(\Hom(L,U_K))$ is also closed in $B$. We also know $N_\G(\Hom(L,U_K))$ is of finite index in $B$. Therefore we can write $$B=\bigcup\limits_{k=1}^n  b_k N_{\G}(\Hom(L,U_K))$$ for some $b_k \in B,$ and it follows that $B$ is closed in $\Hom(L,U_K)$. Now recall that a normal subgroup $H$  of a topological group $G$, is open if and only if the quotient topology on $G/H$ is discrete. So since $M_K=\ZZ$ we must have that $U_K$ is an open subgroup of $C_K$ since we know that $U_K$ is closed in $C_K$ and  $C_K/ U_K \cong \ZZ.$ Therefore  $\Hom_{\G}(L,U_K)$ is an open subgroup of $\Hom_{\G}(L,C_K)$. To finish the proof that $N_{\G}(\Hom(L,C_K))$ is closed in $\Hom_{\G}(L,C_K)$ we can use Lemma \ref{buzz}, by letting $$A=\Hom_{\G}(L,C_K), \qquad  H=\Hom_{\G}(L,U_K), \qquad  S=N_\G(\Hom(L,C_K)),$$  and noting that we are in the situation of Lemma \ref{buzz}  since $$B=N_\G(\Hom(L,C_K)) \cap \Hom(L,U_K) = N_\G(\Hom(L,C_K)) \cap \Hom_{\G}(L,U_K),$$ is closed in $H$.
Hence in the $\ZZ$-case we have that $N_\G(\Hom(L,C_K))$ is closed in $\Hom_{\G}(L,C_K)$.

\subsubsection{\bf $\RR$-case} Here we are in a slightly easier situation, since in this case the exact sequence $$ 1 \lra U_K \lra C_K \lra \RR^{>0} \lra 1$$ splits as a sequence of $\G$-modules. Therefore the sequence 
\[0 \lra \Hom(L,U_K) \overset{\lambda}\lra \Hom(L,C_K) \overset{\mu}\lra \Hom(L,M_K) \lra 0\] also splits as a sequence of $\G$-modules, so  we get   \begin{equation} \Hom(L,C_K)=\Hom(L,U_K) \times \Hom(L,\RR),\tag{\ddag} \end{equation}
and $$N_{\G}(\Hom(L,C_K))=N_{\G}(\Hom(L,U_K)) \times N_{\G}(\Hom(L,\RR)).$$
Furthermore, if we look at the zeroth cohomology groups of $(\ddag)$ , we get $$\Hom_{\G}(L,C_K)=\Hom_{\G}(L,U_K) \times \Hom_{\G}(L,\RR). $$

\begin{prop}\label{p15} In the $\RR$-case we have that $\H^{0}(\G,\Hom(L,M_K))=0$.
	
\end{prop}

\begin{proof} First note that $$\H^{0}(\G,\Hom(L,M_K)=\H^{0}(\G,\Hom(L,\RR))=\Hom_{\G}(L,\RR)/ N_{\G}( \Hom(L,\RR))$$ so the result will follow if we can show that any $\G$-linear homomorphism from $L$ to $\RR$ can be written as the norm of some homomorphism from $L$ to $\RR$. Now since $\G$ acts trivially on $\RR$,  we see that for any $\varphi \in \Hom_{\G}(L,\RR)$  we have $N_{\G}(\varphi) = m\varphi$ where $m = |\G|.$ Therefore, since $\RR$ is a divisible group, $\theta=(1 / m) \varphi$ is also a homomorphism from $L$ to $\RR$,  and thus $\varphi = N_{\G}(\theta)$, which gives the result.

\end{proof}

Now the proposition above tells us that $ N_{\G}(\Hom(L,\RR)) = \Hom_{\G}(L,\RR)$  and, as before,  $N_{\G}(\Hom(L,U_K))$ is  closed in $\Hom_{\G}(L,U_K)$.
It then follows that $$N_{\G}(\Hom(L,U_K)) \times N_{\G}(\Hom(L,\RR))$$ is closed in $$\Hom_{\G}(L,U_K) \times \Hom_{\G}(L,\RR).$$
Hence $N_{\G}(\Hom(L,C_K))$ is closed in $\Hom_{\G}(L,C_K)$.

So we have shown that in both the $\ZZ$-case and $\RR$-case $N_{\G}(\Hom(L,C_K))$ is closed in $\Hom_{\G}(L,C_K)$ and of finite index. Thus $N_{\G}(\Hom(L,C_K))$ is open in $\Hom_{\G}(L,C_K)$, which proves Proposition \ref{p10}.
\begin{prop}\label{4.16} A 1-cocycle in $Z^{1}(\W,\wh{T}_D)$ is continuous if and only if its restriction to $C_K$ is continuous.
\end{prop}

\begin{proof}Clearly if $x \in Z_{cts}^{1}(\W,\wh{T}_D)$, then its restriction to $C_K$ will also be continuous, so we only need to prove the other direction.
	
	If $x \in Z^{1}(\W,\wh{T}_D)$ is continuous on $C_K$, define $\s(a)=x(wa)$, where $w \in \W$, $a \in C_K$, then in order to prove that $x$ is continuous on the coset $wC_K$ we only need to prove that $\s$ is continuous. Now, since $x$ is a 1-cocycle, we have $$\s(a)=x(wa)=wx(a)+x(w).$$ So as $a$ goes through $C_K$, we have that $wx(-)$ is continuous since $x(-)$ is continuous on $C_K$ and the action of $\W$ is continuous (since it is induced from the continuous action of $\G$ on $\wh{T}_{D}$). Therefore, since $x(w)$ is just a constant,  $\s$ is continuous and hence $x$ is continuous on $wC_K$. From this it follows that $x$ is continuous on all of $\W$.
	
\end{proof}

Now observe that we have the following diagram:
$$
\xymatrix@C=2em@R=2em{
	H^{1}(\W,\wh{T}_D) \ar[d]^{\Res} \ar[r]^-{\sim} & \Hom(H_{1}(\W,\L),D) \ar@{->>}[d]^{\Cor'}\\
	H^{1}(C_K,\wh{T}_D)  \ar[r]^-{\sim} & \Hom(H_{1}(C_K,\L),D),}
$$
where the horizontal arrows are isomorphisms given by Proposition \ref{p3}, $\Res$ is the standard restriction map on cohomology groups, and $\Cor'$ is the surjective map induced from $$\Cor:H_{1}(C_K,\L) \lra H_{1}(\W,\L).$$ Now it is easy to show this diagram is in fact commutative since if we take a 1-cocyle $f \in  Z^{1}(\W,\wh{T}_D) $ and first move along the top of the diagram, then by Remark \ref{rmkiso} and the definition of $\Cor$ for homology, we get that $f$ maps to the homomorphism sending $x \in H_{1}(C_K,\L)$ to $$\sum_{a \in C_K} \langle f(a) , x(a) \rangle. $$  But this is clearly the same as going around the diagram in the other direction.

This diagram together with Proposition \ref{p10} and Proposition \ref{4.16},  reduce Proposition \ref{continu} to proving the following:
\begin{prop}
	Let $$\Psi:H^{1}(C_K,\wh{T}_D) \overset{\sim}\lra  \Hom(H_{1}(C_K,\L),D).$$ If $f$ is a 1-cocycle in $Z^{1}(C_K,\wh{T}_D)$, then $f$ is continuous if and only if $\Psi([f])$ is a continuous homomorphism in $\Hom(H_{1}(C_K,\L),D)$ or, what is the same, in $\Hom(\L \otimes C_K,D)$.
\end{prop}

\begin{proof}
	In this case we can see what $f$ maps to. It will correspond to the homomorphism  $\Psi([f]) \in \Hom(\L \otimes C_K, D)$
	$$
	\Psi([f]):\wh{\l} \otimes a \longmapsto \langle \wh{\l} , f(a) \rangle$$ where, in this case, we have that $\langle-,- \rangle$ is the natural bilinear mapping $\langle -,-\rangle : \L \times \wh{T}_D \to D.$ This bilinear map can easily be seen to be continuous by observing that $\L$ has the discrete topology and $\wh{T}_{D}$ has topology induced by that of $D$.
	Then with this it is clear that $f$ will be a continuous 1-cocycle if and only if $\Psi(\bar{f})$ is a continuous homomorphism.
\end{proof}
Hence we have proven Proposition \ref{continu}, which completes the proof of  Theorem \ref{t1}.

\section{\bf Applications to algebraic tori}

Our next goal is to use this to say in a bit more detail how Theorem \ref{t1} relates to algebraic tori. Recall that at the start we identified $T(K)$ (the group of $K$-rational points) with $\Hom(L,K^{\times})$. In the case that $K$ is a global field, we have an exact sequence $$1 \lra K^{\times} \lra \AA_{K}^{\times} \lra C_K \lra 1,$$ which,  since $L$ is free (and hence projective),  gives us the exact sequence
\[\xymatrix@C=2em@R=1em{
	0 \ar[r] & \Hom(L,K^{\times}) \ar[r] & \Hom(L,\AA_{K}^{\times}) \ar[r] &  \Hom(L,C_K) \ar[r] & 0.\\
	& T(K) \ar@2{=}[u] & T(\AA_{K}) \ar@{=}[u] & &
	&&&&}\]
Hence we can identify $T(\AA_K) / T(K)$ with $\Hom(L,C_K)$. Also from this we obtain a long exact sequence $$0 \lra T(F) \lra T(\AA_F) \lra \Hom_{\G}(L,C_K) \lra H^{1}(\G,T(K)) \lra \cdots, $$ which we will use later. Next we need a result that allows us to switch between $\W$ and $W_F$.

\begin{prop}\label{p16}
	Let $D$ be a Hausdorff divisible abelian topological group. Then the inflation map $$\Inf: H_{cts}^{1}(\W,\wh{T}_D) \lra H_{cts}^{1}(W_F,\wh{T}_D)$$ is bijective. 
\end{prop}

\begin{proof}(See \cite[p.\ 111]{milneadt})  First recall that the inflation map is always injective and so is its restriction to $H_{cts}^{1}$, so we only need to prove that in this case it is also surjective.
	So take a continuous 1-cocycle $f : W_F \to \wh{T}_D,$ this will restrict to a continuous 1-cocycle in $Z_{cts}^{1}(W_K,\wh{T}_D)$. Now the kernel of $f$ must contain the commutator group of $W_K$ since $\wh{T}_{D}$ is commutative. Furthermore, since $\wh{T}_D$ is Hausdorff, the kernel of $f$ must also be closed. So $f$ must be trivial on $W_{K}^{c}$ (the closure of the commutator group) and hence it must factor through $W_{F} / W_{K}^{c}=\W$, giving that $\Inf$ is also surjective.

\end{proof}

Recall that for $F$ a global field, we say an element in $H_{cts}^{1}(W_F,\wh{T}_{D})$ is locally trivial if it restricts to zero in $H_{cts}^{1}(W_{F_v},\wh{T}_{D})$ for all places $v$ of $F$.

\begin{thm}\label{t2}
	\begin{enumerate}[(a)]
		\item If $K$ is a local field, then  $H_{cts}^{1}(\W,\wh{T}_D)$ is isomorphic to $ \Hom_{cts}(T(F),D)$.
		\item Let $D'$ be a divisible abelian topological group such that for any finite group $G$, $\Hom(G,D')$ is finite, and let $$\wh{T}_{D'}=\Hom(\L,D').$$ 
		If $K$ is a global field, then there is a canonical  surjective homomorphism $$\H_{cts}^{1}(\W,\wh{T}_{D'}) \lra \Hom_{cts}(T(\AA_F) /T(F) , D'),$$and  the kernel of this homomorphism is finite. Furthermore, if $D'$ is also Hausdorff, then the kernel consists of the locally trivial classes.

	\end{enumerate}

\end{thm}
\begin{proof}
	
	\begin{enumerate}[(a)]
		\item This follows immediately from Theorem \ref{t1}, since in this case $C_K=K^{\times}$ and therefore $T(F) =T(K)^{\G}=\Hom_{\G}(L,C_K)$, and the isomorphism in question is exactly the one we found in Theorem \ref{t1}. 
		
		\item We begin by considering the following commutative diagram
		
		\[
		\xymatrix@C=2em@R=3em{
			0 \ar[r]&T(K) \ar[r]& T(\AA_K) \ar[r] \ar[d]^{N_\G}  &  \Hom(L,C_K) \ar[r] \ar[d]^{N_\G} & 0.\\
			0 \ar[r] & T(F) \ar[r]  &T(\AA_{F}) \ar[r] & \Hom_{\G}(L,C_K) \ar[r] & \cdots}
		\]
		From before, we already know both rows are exact and the commutativity of the diagram tells us that we must have $N_{\G}(\Hom(L,C_K))$ contained in $T(\AA_F) /T(F)$. By Proposition \ref{p11} we know that $N_{\G}(\Hom(L,C_K))$ has finite index in $\Hom_{\G}(L,C_K)$, so it follows that  $T(\AA_F) /T(F)$ has finite index in $\Hom_{\G}(L,C_K)$. Therefore we have the following exact sequence \[0 \lra T(\AA_F) /T(F) \lra \Hom_{\G}(L,C_K) \lra G \lra 0\] where $G$ is some finite group. If we then apply the functor $\Hom(-,D')$ (which is an exact functor since $D'$ is  $\ZZ$-injective) we get the exact sequence \begin{center}\vspace{-0.5cm}
			\adjustbox{max width=\linewidth}{%
				\begin{tikzcd} 0 \to \Hom(G,D') \to \Hom(\Hom_{\G}(L,C_K),D') \to  \Hom(T(\AA_F) /T(F), D') \to 0 .\end{tikzcd}
			}	
		\end{center}
		Now by assumption  $\Hom(G,D')$ is a finite group, so it follows by Theorem \ref{t1} that we have a  homomorphism from $H_{cts}^{1}(\W,\wh{T}_{D'})$ onto $\Hom_{cts}(T(\AA_F) /T(F),D')$ with finite kernel.
		
		Next we want prove that if $D'$ is Hausdorff, then the kernel consists of locally trivial classes. With this in mind we use \ref{p16} and \ref{weill} to form the commutative diagram
		$$
		\xymatrix@C=4em@R=3em{
			H_{cts}^{1}(W_{F},\wh{T}_{D'})  \ar[r] \ar[d] & \Hom_{cts}(T(\AA_F) /T(F) , D') \ar[d] \\
			\prod_{v} H_{cts}^{1}(W_{F_v},\wh{T}_{D'})  \ar[r]^-{\sim} & \prod_{v} \Hom_{cts}(T(F_{v}),D'). \\
		}$$
		From part (a) we have that the lower horizontal arrow will be an isomorphism and the result will follow if we can prove that the second vertical arrow is injective.
		
		Let $\chi: T(\AA_F) /T(F) \lra D'$ be a continuous homomorphism, whose restriction to $T(F_{v})$ is the trivial homomorphism for all places $v$. We want to show that this is in fact the trivial homomorphism. Note that we have  $$\bigoplus_v T(F_v) \subseteq \ker \chi.$$ Now since $D'$ is Hausdorff, we have that $\ker \chi$ is closed in $T(\AA_F)/T(F)$. Therefore the result will follow if we can show that $\bigoplus_v T(F_v) $ is dense in $T(\AA_F)$, since this would mean that $T(\AA_F)$ is also in the kernel of $\chi$, which makes $\chi$ trivial. Now, since $T(\AA_F)$ is the restricted topological product of $T(F_v)$ with respect to $T(\mathcal{O}_v)$, it is easy to see that any non-empty open set in $T(\AA_F)$ meets $\bigoplus_v T(F_v)$. Hence the result follows.
		
	\end{enumerate}

\end{proof}

If we now take $D=\CC_p^\times$ then the above theorem relates admissible homomorphisms $\W \to \wh{T}_{\CC_p^\times} \rtimes \G$ to continuous representations of $T$ on the "$p$-adic Banach space" $\CC_p^\times$. This can be seen as the $\GL_1$ case of more general conjectures of the $p$-adic Langlands programme for $\GL_n$ which link admissible $n$-dimensional representations of $\Gal(\overline{\QQ}_p/\QQ_p)$ to certain continuous $n$-dimensional representations of $\GL_n$ on $p$-adic Banach spaces (cf. \cite[p. 116]{milneadt}). Specifically, if one takes a more general reductive group $G$ with Langlands dual group ${}^LG$, then the Langlands program links admissible homomorphisms $\W \to {}^LG$ with irreducible automorphic representations of $G$ (which are the analogue of our group $\Hom_{cts}(T(F),D)$). For more details and survey of what is known about the $p$-adic Langlands programme we refer the reader to \cite{BreuEmerg}.

\bibliographystyle{alpha}

\bibliography{bibli}

\begin{thebibliography}{Wei95}

\bibitem[AT67]{artintate2}
E.~Artin and J.T. Tate.
\newblock {\em Class field theory}.
\newblock Ams Chelsea Publishing. AMS Chelsea Publishing; 2nd Revised edition
  edition, 1967.

\bibitem[Bre10]{BreuEmerg}
Christophe Breuil.
\newblock The emerging {$p$}-adic {L}anglands programme.
\newblock In {\em Proceedings of the {I}nternational {C}ongress of
  {M}athematicians. {V}olume {II}}, pages 203--230. Hindustan Book Agency, New
  Delhi, 2010.

\bibitem[Bro82]{brown}
K.S. Brown.
\newblock {\em Cohomology of Groups}.
\newblock Graduate Texts in Mathematics. Springer, 1982.

\bibitem[Lan97]{lang}
R.~P. Langlands.
\newblock Representations of abelian algebraic groups.
\newblock {\em Pacific J. Math.}, (Special Issue):231--250, 1997.
\newblock Olga Taussky-Todd: in memoriam.

\bibitem[Mil06]{milneadt}
J.S. Milne.
\newblock {\em Arithmetic Duality Theorems}.
\newblock BookSurge, 2006.

\bibitem[NS99]{neukalg}
J.~Neukirch and N.~Schappacher.
\newblock {\em Algebraic Number Theory}.
\newblock Grundlehren der mathematischen Wissenschaften. Springer, 1999.

\bibitem[Rot09]{rotmanhom}
J.J. Rotman.
\newblock {\em An Introduction to Homological Algebra}.
\newblock Universitext - Springer-Verlag. Springer Science+Business Media, LLC,
  2009.

\bibitem[SG80]{serrelocal}
J.P. Serre and M.J. Greenberg.
\newblock {\em Local Fields}.
\newblock Graduate Texts in Mathematics. Springer, 1980.

\bibitem[Tat79]{tatewg}
J.~Tate.
\newblock Number theoretic background.
\newblock In {\em Proc. Symp. Pure Math}, volume~33, pages 3--26, 1979.

\bibitem[Wei69]{weiss}
E.~Weiss.
\newblock {\em Cohomology of Groups}.
\newblock Pure and applied mathematics. Academic Press, 1969.

\bibitem[Wei95]{weib}
C.A. Weibel.
\newblock {\em An Introduction to Homological Algebra}.
\newblock Cambridge Studies in Advanced Mathematics. Cambridge University
  Press, 1995.

\end{thebibliography}

\end{document}